\documentclass[english]{amsart}
\usepackage[latin9]{inputenc}
\usepackage{geometry}
\usepackage{amsthm}
\usepackage{amstext}
\usepackage{amssymb}
\usepackage{graphicx}
\usepackage{setspace}
\usepackage{esint}
\makeatletter
\numberwithin{figure}{section}
\theoremstyle{plain}
\newtheorem{thm}{\protect\theoremname}
  \theoremstyle{plain}
  \newtheorem{prop}[thm]{\protect\propositionname}
  \theoremstyle{plain}
  \newtheorem{cor}[thm]{\protect\corollaryname}
  \theoremstyle{plain}
  \newtheorem{lem}[thm]{\protect\lemmaname}
  \theoremstyle{remark}
  \newtheorem{rem}[thm]{\protect\remarkname}
  \theoremstyle{remark}
  \newtheorem*{rem*}{\protect\remarkname}

\date{}
\pdfpageattr{/Group << /S /Transparency /I true /CS /DeviceRGB>>}
\usepackage{ae,aecompl} 
\usepackage{amssymb}
\usepackage{enumerate}
\usepackage{bbm}
\usepackage{hyperref}
\hypersetup{
                    colorlinks=true,
                    linkcolor=red,
                    citecolor=blue,
                    urlcolor=green,
                    linktocpage=true,
                    pdfstartview=FitH,
                }

\newtheorem*{facts}{{\bf Facts}}

\@ifundefined{showcaptionsetup}{}{%
 \PassOptionsToPackage{caption=false}{subfig}}
\usepackage{subfig}
\makeatother

\usepackage{babel}
  \providecommand{\corollaryname}{Corollary}
  \providecommand{\lemmaname}{Lemma}
  \providecommand{\propositionname}{Proposition}
  \providecommand{\remarkname}{Remark}
\providecommand{\theoremname}{Theorem}

\begin{document}
\global\long\def\R{\mathbb{R}}

\global\long\def\C{\mathbb{C}}

\global\long\def\Z{\mathbb{Z}}

\global\long\def\N{\mathbb{N}}

\global\long\def\Q{\mathbb{Q}}

\global\long\def\T{\mathbb{T}}

\global\long\def\F{\mathbb{F}}

\global\long\def\Sph{\mathbb{S}}

\global\long\def\sub{\subseteq}

\global\long\def\cvx{\mbox{Cvx}\left(\R^{n}\right)}

\global\long\def\cvxo{\text{Cvx}_{0}\left(\R^{n}\right)}

\global\long\def\lcg{\text{LC}_{g}\left(\R^{n}\right)}

\global\long\def\lc{\text{LC}\left(\R^{n}\right)}

\global\long\def\dis{\text{Dis }\left(\R^{n}\right)}

\global\long\def\one{\mathbbm1}

\global\long\def\infc#1#2{#1\,+_{\text{cvx}}\,#2}

\global\long\def\supc#1#2{#1\,+_{\text{lc}}\,#2}

\global\long\def\vol{\text{Vol }}

\global\long\def\EE{\mathbb{E}}

\global\long\def\cvdold{\,\cdot_{\text{cvx}}\,}

\global\long\def\lcdold{\,\cdot_{\text{lc}}\,}

\global\long\def\cvd{\,\check{\cdot}\,}

\global\long\def\lcd{\,\hat{\cdot}\,}

\global\long\def\cvxdot{\,\cdot_{\text{cvx}}\,}

\global\long\def\cvxplus{\,+_{{\scriptstyle {\rm cvx}}}\,}

\global\long\def\convp{\overset{\smallsmile}{+}}

\global\long\def\logcp{\overset{\smallfrown}{+}}

\global\long\def\cvp{\,+_{\text{cvx}}\,}

\global\long\def\lcp{\,+_{\text{lc}}\,}

\global\long\def\cvone#1{\one_{\left\{  #1\right\}  }^{\infty}}

\global\long\def\cvxone#1{\one_{ #1 }^{\infty}}

\global\long\def\epi#1{\text{epi}\left\{  #1\right\}  }

\global\long\def\sp{{\rm sp}}

\global\long\def\K{\mathcal{K}}

\global\long\def\A{\mathcal{A}}

\global\long\def\L{\mathcal{L}}

\global\long\def\P{\mathcal{P}}

\global\long\def\W{\mathcal{W}}

\global\long\def\iprod#1#2{\langle#1,\,#2\rangle}

\global\long\def\cvrs{{\rm Cvrs}}

\global\long\def\cvrsb{\overline{{\rm Cvrs}}}

\global\long\def\uball{B_{2}^{n}}

\global\long\def\conv{{\rm Conv}}

\global\long\def\Gauss{\gamma_{n}^{\sigma}}

\global\long\def\EXP#1{\EE_{{\textstyle #1}}}

\global\long\def\eps{\varepsilon}

\global\long\def\PP{\mathbb{P}}

\global\long\def\IndUball{\one_{B_{2}^{n}}}

\global\long\def\J{{\cal J}}

\title{a note on Santal\'{o} inequality for the polarity transform and
its reverse }

\author{Shiri Artstein-Avidan}

\author{Boaz A. Slomka}

\address{School of Mathematical Science, Tel-Aviv University, Ramat Aviv,
Tel Aviv, 69978, Israel}

\email{shiri@post.tau.ac.il (Artstein-Avidan), boazslom@post.tau.ac.il (Slomka)}

\thanks{This work was supported by ISF grant No. 247/11.}


\begin{abstract}
We prove a Santal\'{o} and a reverse Santal\'{o} inequality for
the class consisting of even log-concave functions attaining their
maximal value $1$ at the origin, also called even geometric log-concave
functions. We prove that there exist universal numerical constant
$c,C>0$ such that for any even geometric log-concave function $f=e^{-\varphi}$,
\[
c^{n}\cdot\left|B_{2}^{n}\right|^{2}\le\int_{\R^{n}}e^{-\varphi}\int_{\R^{n}}e^{-\varphi^{\circ}}\le\left(\left|B_{2}^{n}\right|n!\right)^{2}\left(1+\frac{C}{n}\right)
\]
where $B_{2}^{n}$ is the Euclidean unit ball of $\R^{n}$ and $\varphi^{\circ}$
is the polar function of $\varphi$ (not the Legendre transform!),
a transform which was recently re-discovered by Artstein-Avidan and
Milman, and is defined below. The bounds are sharp up to the optimal
constants $c,C$. 
\end{abstract}
\maketitle

\section{Introduction and main results}

Let $\R^{n}$ denote the  Euclidean $n$-dimensional real space, equipped
with the standard scalar product $\iprod{\cdot}{\cdot}$ and the standard
Euclidean norm $\left|x\right|=\iprod xx^{1/2}$. Denote its unit
ball by $B_{2}^{n}$. Let $K\sub\R^{n}$ be a convex body (i.e., compact
and containing the origin $0$ in its interior). The polar set of
$K$ is given by 
\[
K^{\circ}=\left\{ x\in\R\,:\,\sup_{y\in K}\iprod xy\le1\right\} .
\]
The set $K^{\circ}$ is also a convex body, and if $K$ is centrally
symmetric, so is its polar $K^{\circ}$. Denote the Lebesgue volume
of $K$ by $\left|K\right|$. Recall the classical Santal\'{o} inequality
and its reverse: there exists an absolute constant $c>0$ such that
for any dimension $n$ and every centrally symmetric convex body $K\sub\R^{n}$,
\[
c^{n}\cdot\left|B_{2}^{n}\right|^{2}\le\left|K\right|\left|K^{\circ}\right|\le\left|B_{2}^{n}\right|^{2}.
\]
The right hand side is referred to as Santal\'{o} inequality and
is due to Santal\'{o} \cite{Santalo49} (for a simpler proof see
\cite{MeyerPajor90}). The left hand side inequality is referred to
as the reverse Santal\'{o} inequality, or the Bourgain-Milman inequality
and is due to J. Bourgain and V. Milman \cite{BourgainMilman87}.
In fact, the left hand side inequality holds also for convex bodies
which are not centrally symmetric. While bounds for the numerical
constant $c$ have been improving over the years (see \cite{Kuperberg08,NazarovBM})
it is still an open question whether among centrally symmetric convex
bodies the cube is a minimizer for this product, often called the
Mahler product, and whether among general convex bodies a simplex
is a minimizer.\\

Functional versions of Santal\'{o} inequality and its reverse were
also established; A function $f:\R^{n}\to\R$ is said to be log-concave
if it is of the form $e^{-\varphi}$ where $\varphi:\R^{n}\to\R\cup\left\{ \infty\right\} $
is convex. One interpretation for the ``polar'' function of $e^{-\varphi}$
may be the function $e^{-\L\varphi}$ where $\L$ is the well-known
Legendre transform, given by $\L\varphi\left(x\right)=\sup_{y\in\R^{n}}\left[\iprod xy-\varphi\left(y\right)\right]$.
Note that $e^{-\L\varphi}$ is always log-concave. It was proven that
there exists an absolute constant $c>0$ such that for any even log-concave
function $f=e^{-\varphi}$ on $\R^{n}$ with $0<\int f<\infty$, 
\[
\left(\frac{2\pi}{c}\right)^{n}\le\int_{\R^{n}}e^{-\varphi}\cdot\int_{\R^{n}}e^{-\L\varphi}\le\left(2\pi\right)^{n}.
\]
The right hand side is due to Ball \cite{BallPhD}, see also Artstein,
Klartag, and Milman \cite{AKM2005}. In the latter paper it is actually
proved that Gaussians are the only maximizers for this product. For
further related results see also \cite{FradeliziMeyer07,Lehec-Direct09,Lehec-Partitions09}.
The left hand side is due to Klartag and Milman \cite{Klartag2005}.
Note that Santal\'{o} inequality and its reverse for centrally symmetric
convex bodies may be recovered by plugging in $f=e^{-\left\Vert \cdot\right\Vert _{K}^{2}/2}$,
where $\left\Vert \cdot\right\Vert _{K}$ denotes the norm with unit
ball $K$. In the sequel we shall often use $\left\Vert \cdot\right\Vert _{K}$
to denote the gauge function associated with a convex body $K$, even
in the non-symmetric case, so long as the origin belongs to the interior
of $K$. That is, $\left\Vert \cdot\right\Vert _{K}=\inf\left\{ r>0\,\,:\,\, x\in rK\right\} $.

In this note, we prove a different functional version of Santal\'{o}
inequality and its reverse, for an important subclass of (even) log-concave
functions, called\textit{ (even) geometric log-concave functions},
and consisting of all functions of the form $f=e^{-\varphi}$ where
$\varphi$ is a non-negative (even) lower semi-continuous convex function
with $\varphi\left(0\right)=0$. A few years ago, Artstein-Avidan
and Milman \cite{Artstein-Milman2010,AM2011} have proven that, on
this class, up to trivial obvious modifications, there exist exactly
two order-reversing involutions, one is the Legendre transform, and
the other, which we shall call here the polarity transform, is defined
by $\varphi\mapsto\varphi^{\circ}$ where 
\[
\varphi^{\circ}\left(x\right)=\sup_{y\in\R^{n}}\frac{\iprod xy-1}{\varphi\left(y\right)}.
\]
Here we agree that $\frac{+}{0}=\infty$, $\frac{0}{0}=0$ and $\frac{-}{0}=\left(\frac{-}{0}\right)_{+}=0$.
The latter is needed only for this definition to work with the identically
$0$ function. It seems that this transform has appeared only once
in the literature (Rockafeller's book \cite[p. 136]{Rockafellar-Book})
before it was re-discovered in \cite{Artstein-Milman2010,AM2011}.
Among other things, in the latter paper strong reasons were given
to explain why this new transform should be considered as the natural
extension of the notion of polarity from convex bodies to the class
of geometric log-concave functions, and the Legendre transform as
the natural extension of the support function. 

Since these discoveries, the first named author was asked in several
occasions whether a Santal\'{o} type inequality holds also for the
polarity transform. In this note we answer this question in the affirmative.
We prove the following:
\begin{thm}
\label{thm:Lc-Santalo}Let $f=e^{-\varphi}$ be an even geometric
log-concave function on $\R^{n}$ such that $0<\int e^{-\varphi}<\infty$.
Then 
\[
c^{n}\cdot\left|B_{2}^{n}\right|^{2}\le\int_{\R^{n}}e^{-\varphi}\cdot\int_{\R^{n}}e^{-\varphi^{\circ}}\le\left(\left|B_{2}^{n}\right|n!\right)^{2}\left(1+\frac{C}{n}\right)
\]
where $c,C>0$ are universal constants independent of $n$ and $\varphi$.
The left hand side inequality holds also without the assumption that
$\varphi$ is even. 
\end{thm}
\noindent Note that, up to optimal constants $c$ and $C$, one cannot
hope for better bounds since for indicators $\one_{K}$ of centrally
symmetric convex bodies $K\sub\R^{n}$ we recover the Bourgain-Milman
lower bound and for $\varphi=e^{-\left|x\right|}$ the above product
equals $n!^{2}\cdot\left|B_{2}^{n}\right|^{2}$. However, although
the product $\int\varphi\int\varphi^{\circ}$ is invariant under invertible
linear transformations, as one can check that $\left(\varphi\circ A\right)^{\circ}=\varphi^{\circ}\circ A^{-T}$,
it does not imply that this product has maximizers or minimizers.
We remark also that the fixed points of the mapping $\varphi\mapsto\varphi^{\circ}$
in dimension $1$ have been classified by L. Rotem in \cite{RotemFixed}.
The motivation for classifying these fixed points was that in the
convex body case the maximizer for the Mahler product is the unique
fixed point of the polarity transform for convex bodies, the Euclidean
ball. Thus the maximizer of the product considered in Theorem \ref{thm:Lc-Santalo},
if exists, might also be a fixed point, although this was not verified.
We do know to show, by symmetrizations, that maximizers, if exist,
must be rotationally invariant.

This note is organized as follows. In Section \ref{sec:Comparing-level-sets}
we prove some preliminary facts about log-concave functions, in particular
that the class of geometric log-concave functions with (finite) positive
integral is closed under the polarity transform, and a simple yet
useful generalization of a proposition which appeared in \cite{Milman-Rotem2012}
about the connection between level sets of a function and its polar.
The proof of Theorem \ref{thm:Lc-Santalo} is split into two propositions,
the right hand side inequality is proven in Section \ref{sec:Santaloineq}
and the left hand side inequality is proven in Section \ref{sec:ReverseSantaloIneq}.

\section{Preliminary facts}

\subsection{Comparing level sets of a function and its polar\label{sec:Comparing-level-sets}}

Let $f=e^{-\varphi}$ be a geometric log-concave function on $\R^{n}$.
We adopt the notation of \cite{Milman-Rotem2012} and denote the level
sets of $f$ by 
\[
\overline{K}_{t}\left(f\right):=\left\{ x\in\R^{n}\,:\, f\left(x\right)\ge t\right\} 
\]
for any $0<t\le1$. Similarly we denote the level sets of $\varphi$
by
\[
\underline{K}_{t}\left(\varphi\right):=\left\{ x\in\R^{n}\,:\,\varphi\left(x\right)\le t\right\} 
\]
for any $0\le t<\infty$. Note that $\overline{K}_{t}\left(f\right)$
is a closed convex set, and also that $\overline{K}_{t}\left(f\right)=\underline{K}_{\ln\left(1/t\right)}\left(\varphi\right).$
Moreover, note that if $f$ is even then $\overline{K}_{t}\left(f\right)$
are centrally symmetric. Denote the ray emanating from the origin
and passing through $x\in\R^{n}$ by $\R_{+}x=\left\{ \alpha x\,:\,\alpha\ge0\right\} .$
The following proposition is a simple yet useful generalization of
a proposition of V. Milman and L. Rotem  \cite[Proposition 11]{Milman-Rotem2012}.
\begin{prop}
\label{prop:level-sets}For any geometric convex function $\varphi:\R^{n}\to\left[0,\infty\right]$
and any $s,t>0$ we have that
\[
\left(\underline{K}_{1/s}\left(\varphi\right)\right)^{\circ}\sub\underline{K}_{s}\left(\varphi^{\circ}\right)\sub\left(st+1\right)\left(\underline{K}_{t}\left(\varphi\right)\right)^{\circ}.
\]
Moreover, if $\left(\underline{K}_{1/s}\left(\varphi\right)\right)^{\circ}=\underline{K}_{s}\left(\varphi^{\circ}\right)$
for all $s>0$ then 
\[
\varphi=\begin{cases}
0, & x\in L\\
\left\Vert x\right\Vert _{K}, & x\not\in L
\end{cases}
\]

\noindent for some, appropriate, convex sets $K,L$ containing the
origin. Equivalently, the restriction of $\varphi$ to every ray $\R_{+}x$
is either linear or a convex indicator $\one_{\left[0,a\right]}^{\infty}$.
Furthermore, if $\varphi$ is even, then either $\varphi=\one_{K}^{\infty}$
or $\varphi=\left\Vert \cdot\right\Vert _{K}$ for some convex set
$K$. 

\end{prop}
\noindent For $t=1/s$ Proposition \ref{prop:level-sets} reads $\left(\underline{K}_{1/s}\left(\varphi\right)\right)^{\circ}\sub\underline{K}_{s}\left(\varphi^{\circ}\right)\sub2\left(\underline{K}_{1/s}\left(\varphi\right)\right)^{\circ}$
which is exactly \cite[Proposition 11]{Milman-Rotem2012}. For the
reader's convenience, we provide a geometric proof for both inclusions
in Proposition \ref{prop:level-sets} . To this end, we need to recall
the following simple facts about the polarity transform (for proofs
see e.g., \cite{AM2011}):\\
\begin{facts}

\noindent For any geometric convex functions $\varphi,\psi$ on $\R^{n}$
we have that

\begin{enumerate}[1.]

\item$\varphi\le\psi$ if and only if $\psi^{\circ}\le\varphi^{\circ}.$
Also $\left(\varphi^{\circ}\right)^{\circ}=\varphi$.\label{enu:fact 1}

\item$\max\left(\varphi,\psi\right)=\hat{\min}\left(\varphi^{\circ},\psi^{\circ}\right)$
where $\hat{\min}\left(\varphi,\psi\right)=\sup\left\{ \phi\:\,\text{{\rm a geometric convex function}}\,:\,\phi\le\varphi\text{ and }\mbox{\ensuremath{\phi\le\psi}}\right\} $\label{fact 2}.

\item For any convex body $K$ which includes the origin and any
$t>0$, $\left(t\left\Vert \cdot\right\Vert _{K}\right)^{\circ}=\frac{1}{t}\left\Vert \cdot\right\Vert _{K^{\circ}}$
and $\left(\one_{K}^{\infty}\right)^{\circ}=\one_{K^{\circ}}^{\infty}$,
where $\one_{K}^{\infty}=-\log\left(\one_{K}\right)$.\label{fact3}\\

\end{enumerate}

\noindent \end{facts}
\begin{proof}
[Proof of Proposition \ref{prop:level-sets}] Fix $t>0$. Since $\varphi$
is a geometric convex function,  it readily follows that for all $x\in\underline{K}_{t}\left(\varphi\right)$,
\[
\one_{\underline{K}_{t}}^{\infty}\left(x\right)\le\varphi\left(x\right)\le t\cdot\left\Vert x\right\Vert _{\underline{K}_{t}\left(\varphi\right)}
\]
and for all $x\not\in\underline{K}_{t}\left(\varphi\right)$, 
\[
t\cdot\left\Vert x\right\Vert _{\underline{K}_{t}\left(\varphi\right)}\le\varphi\left(x\right)\le\one_{\underline{K}_{t}}^{\infty}\left(x\right).
\]
Thus, the following inequality holds:

\[
\hat{\min}\left(t\cdot\left\Vert \cdot\right\Vert _{\underline{K}_{t}\left(\varphi\right)},\,\one_{\underline{K}_{t}\left(\varphi\right)}^{\infty}\right)\le\varphi\le\max\left(t\cdot\left\Vert \cdot\right\Vert _{\underline{K}_{t}\left(\varphi\right)},\,\one_{\underline{K}_{t}\left(\varphi\right)}^{\infty}\right).
\]
\\
Applying the inequality with the polar transform and using Facts \ref{enu:fact 1}-\ref{fact3}
yields 

\noindent 
\[
\psi_{1}^{t}:=\hat{\min}\left(\frac{1}{t}\cdot\left\Vert \cdot\right\Vert _{\underline{K}_{t}^{\circ}\left(\varphi\right)},\,\one_{\underline{K}_{t}^{\circ}\left(\varphi\right)}^{\infty}\right)\le\varphi^{\circ}\le\max\left(\frac{1}{t}\cdot\left\Vert \cdot\right\Vert _{\underline{K}_{t}^{\circ}\left(\varphi\right)},\,\one_{\underline{K}_{t}^{\circ}\left(\varphi\right)}^{\infty}\right)=:\psi_{2}^{t}.
\]
\\
As illustrated in Fig. \ref{fig:Level-Sets}, 
\begin{figure}[h]
\subfloat[\label{fig:level1} $\underline{K}_{s}^{\circ}\left(\varphi\right)\sub\frac{s+1/t}{1/t}\cdot\underline{K}_{t}\left(\varphi^{\circ}\right)$]{\begin{centering}
\includegraphics[width=6cm]{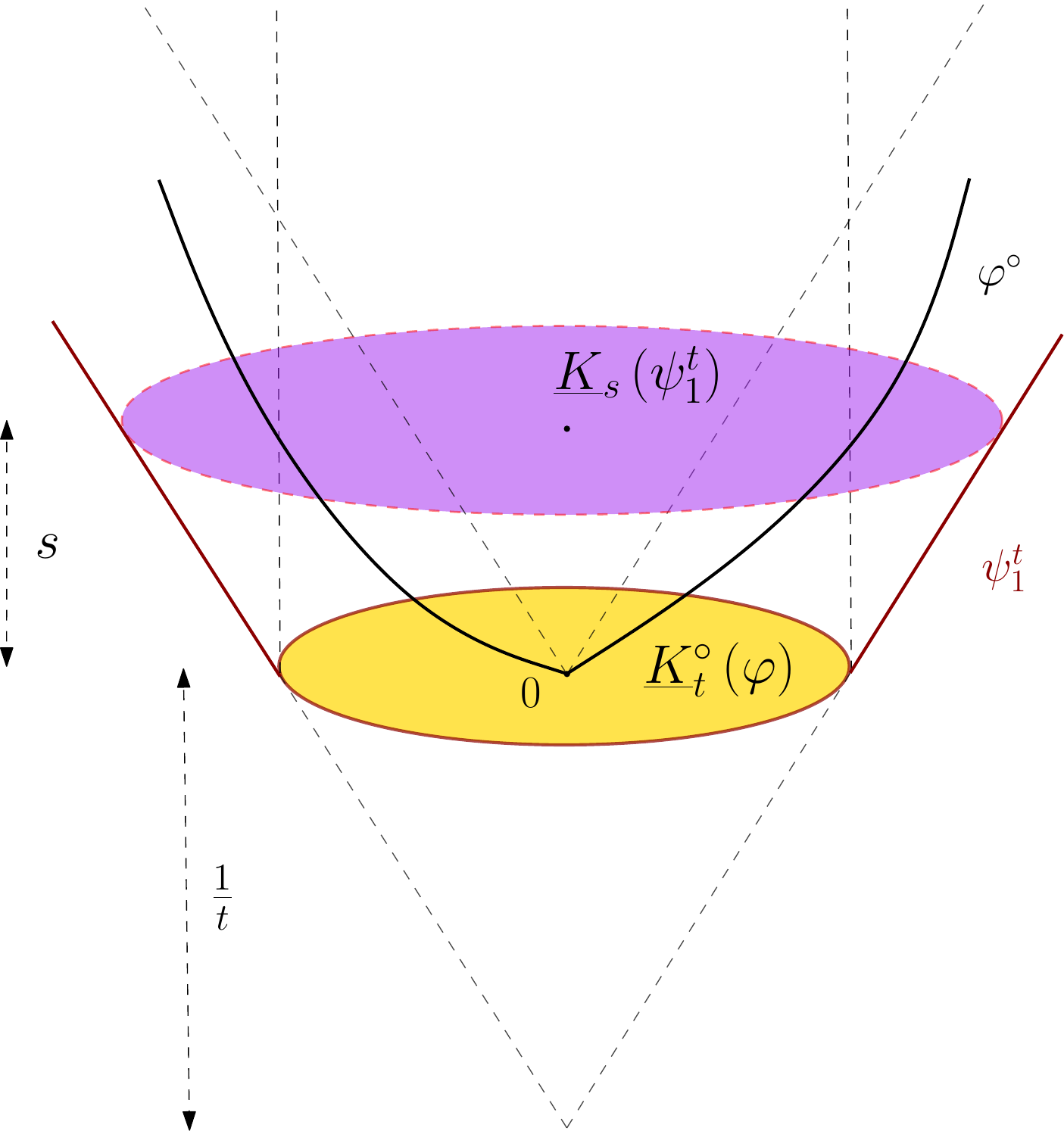}
\par\end{centering}

}\hfill{}\subfloat[$\underline{K}_{1/t}\left(\varphi\right)^{\circ}\sub\underline{K}_{s}\left(\varphi^{\circ}\right)$]{\begin{centering}
\includegraphics[width=6cm]{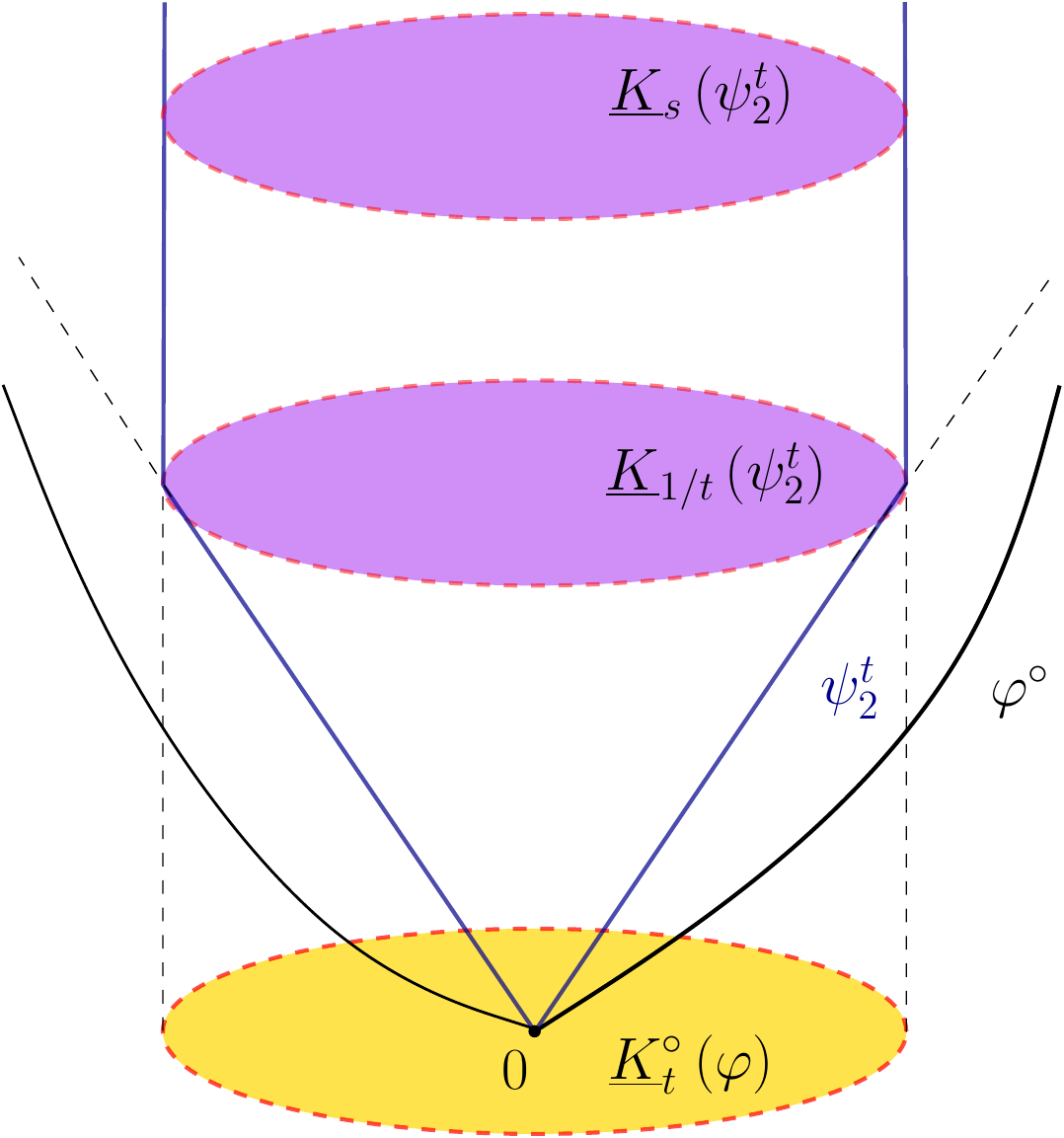}
\par\end{centering}

}

\caption{\label{fig:Level-Sets}Comparing level sets of $\varphi$ and $\varphi^{\circ}$
in Proposition \ref{prop:level-sets}}
\end{figure}
we have that $\underline{K}_{s}\left(\varphi^{\circ}\right)\sub\underline{K}_{s}\left(\psi_{1}^{t}\right)$
and $\underline{K}_{s}\left(\psi_{1}^{t}\right)=\left(st+1\right)\underline{K}_{t}\left(\varphi\right)^{\circ}$
for all $s>0$ and $\underline{K}_{s}\left(\psi_{2}^{t}\right)\sub\underline{K}_{s}\left(\varphi^{\circ}\right)$
and $\underline{K}_{s}\left(\psi_{2}^{t}\right)=\underline{K}_{t}^{\circ}\left(\varphi\right)$
for all $s\ge\frac{1}{t}$, as desired. Equivalently, one may verify
that 
\[
\psi_{1}^{t}\left(x\right)=\begin{cases}
0, & x\in\underline{K}_{t}\left(\varphi\right)^{\circ}\\
\frac{1}{t}\left\Vert x\right\Vert _{\underline{K}_{t}\left(\varphi\right)^{\circ}}-\frac{1}{t}, & x\not\in\underline{K}_{t}\left(\varphi\right)^{\circ}
\end{cases},\,\,\psi_{2}^{t}\left(x\right)=\begin{cases}
\frac{1}{t}\left\Vert x\right\Vert _{\underline{K}_{t}\left(\varphi\right)^{\circ}}, & x\in\underline{K}_{t}\left(\varphi\right)^{\circ}\\
+\infty, & x\not\in\underline{K}_{t}\left(\varphi\right)^{\circ}
\end{cases}
\]
\\
which together with $\psi_{1}^{t}\le\varphi^{\circ}\le\psi_{2}^{t}$,
implies the same desired inclusions. 

Next, we deal with the equality case of the left hand side inequality;
we wish to find all possible forms of a function $\varphi$ that satisfies
\begin{equation}
\left(\underline{K}_{1/s}\left(\varphi\right)\right)^{\circ}=\underline{K}_{s}\left(\varphi^{\circ}\right)\,\,\,\text{for all\,\,\,\ }s>0.\label{eq:EqualityEq}
\end{equation}
 Note that, so far, we have that $\varphi^{\circ}\le\psi_{2}:=\hat{\inf}_{t>0}\psi_{2}^{t}$,
or equivalently $\varphi\ge\psi_{2}^{\circ}=\sup_{t>0}\left(\psi_{2}^{t}\right)^{\circ}$.
First, we prove that $\psi_{2}^{\circ}\equiv\varphi$ if and only
if 
\begin{equation}
\varphi=\begin{cases}
\left\Vert x\right\Vert _{K}, & x\in L\\
+\infty, & x\not\in L
\end{cases}\label{eq:Eqform1}
\end{equation}
for some convex sets $K$ and $L$ containing the origin. Indeed,
it is not hard to check that such $\varphi$'s satisfy the desired
equality. For the other direction, note that this claim is point-wise
since
\[
\left(\psi_{2}^{t}\right)^{\circ}\left(x\right)=\begin{cases}
0, & x\in\underline{K}_{t}\left(\varphi\right)\\
t\left\Vert x\right\Vert _{\underline{K}_{t}\left(\varphi\right)}-t, & x\not\in\underline{K}_{t}\left(\varphi\right)
\end{cases}
\]
and so we may assume without loss of generality that the dimension
is $n=1$. Let us show that there are no other possible forms of $\varphi$.
Assume that $\varphi$ is not of the claimed form. Then without loss
of generality (otherwise take $\varphi\left(-x\right)$), there exist
$0<x_{1}<x_{2}$ such that $\varphi'\left(x_{1}\right)<\varphi'\left(x_{2}\right)$
(and the derivatives exist). Let $0<x_{t}$ denote the point (if exists)
for which $\varphi\left(x_{t}\right)=t$. Then for every $0<t\le\varphi\left(x_{1}\right)$
we have that 
\[
\left(\psi_{2}^{t}\right)^{\circ}\left(x_{2}\right)=\frac{\varphi\left(x_{t}\right)}{x_{t}}\left(x_{2}-x_{t}\right)\le\frac{\varphi\left(x_{1}\right)}{x_{1}}x_{2}<\varphi\left(x_{2}\right),
\]
for every $\varphi\left(x_{1}\right)<t<\varphi\left(x_{2}\right)$
we have that 
\[
\left(\psi_{2}^{t}\right)^{\circ}\left(x_{2}\right)=\frac{\varphi\left(x_{t}\right)}{x_{t}}\left(x_{2}-x_{t}\right)\le\frac{\varphi\left(x_{2}\right)}{x_{2}}\left(x_{2}-x_{1}\right)<\varphi\left(x_{2}\right)
\]
and for every $t\ge\varphi\left(x_{2}\right)$ we have that $\left(\psi_{2}^{t}\right)^{\circ}\left(x_{2}\right)=0$.
Thus, $\psi_{2}^{\circ}\left(x_{2}\right)\le\sup\left(\frac{\varphi\left(x_{1}\right)}{x_{1}}x_{2},\,\frac{\varphi\left(x_{2}\right)}{x_{2}}\left(x_{2}-x_{1}\right)\right)<\varphi\left(x_{2}\right)$,
as claimed. Note that if $\psi_{2}\neq\varphi^{\circ}$ then $\underline{K}_{s_{0}}\left(\psi_{2}\right)\subsetneqq\underline{K}_{s_{0}}\left(\varphi^{\circ}\right)$
for some $s_{0}>0$ and so, 
\[
\underline{K}_{1/s_{0}}\left(\varphi\right)^{\circ}=\underline{K}_{s_{0}}\left(\psi_{2}^{1/s_{0}}\right)\sub\underline{K}_{s_{0}}\left(\psi_{2}\right)\subsetneqq\underline{K}_{s_{0}}\left(\varphi^{\circ}\right).
\]
Therefore, all geometric convex functions $\varphi$ satisfying (\ref{eq:EqualityEq})
are of the form (\ref{eq:Eqform1}) for some convex sets $K$ and
$L$ containing the origin. Moreover, note that $\varphi$ satisfies
(\ref{eq:EqualityEq}) if and only if $\varphi^{\circ}$ satisfies
(\ref{eq:EqualityEq}) and thus it follows that $\varphi^{\circ}$
is of the form (\ref{eq:Eqform1}) as well. However, we have that
\begin{equation}
\varphi^{\circ}=\begin{cases}
0, & x\in L^{\circ}\\
\left\Vert x\right\Vert _{K^{\circ}}\left(1-\frac{1}{\left\Vert x\right\Vert _{L^{\circ}}}\right), & x\not\in L^{\circ}
\end{cases},\label{eq:EqForm2}
\end{equation}
which means that if $\varphi$ satisfies (\ref{eq:EqualityEq}) then
it is simultaneously of the form (\ref{eq:Eqform1}) and of the form
(\ref{eq:EqForm2}) (with different $K,L$ in each form). It is not
hard to check that if $\varphi$ is of both forms then its restriction
to every ray $\R_{+}x$ is either linear or a convex indicator function
$\one_{\left[0,a\right]}^{\infty}$ (the first form excludes the possibility
of attaining $0$ on a segment and then linear, and latter excludes
the possibility of attaining linear values on a segment and then attaining
$+\infty$). Let $\varphi$ be a geometric convex function such that
its restriction to every ray $\R_{+}x$ is either linear or a convex
indicator. Then $\varphi$ is of both form (\ref{eq:Eqform1}) and
of form (\ref{eq:EqForm2}). Indeed, define $L=\left\{ \varphi=0\right\} $,
and $K$ by: 
\[
K\cap\R_{+}x=\begin{cases}
\left\{ 0\right\} , & \exists y\in\R_{+}x\,\,\varphi\left(y\right)=\infty\\
\left\{ x\,:\,\varphi\left(x\right)\le1\right\} \cap\R_{+}x, & o/w
\end{cases}.
\]
Then, one can verify that 
\begin{equation}
\varphi\left(x\right)=\begin{cases}
0, & x\in L\\
\left\Vert x\right\Vert _{K}, & x\not\in L
\end{cases}\label{eq:EqForm3}
\end{equation}
(which is a special case of (\ref{eq:EqForm2}); in every direction
$\varphi$ is either linear or a convex indicator) and that $K$ is
indeed convex (due to the convexity of $\varphi$). Similarly we
may define $L'={\rm Supp}\left\{ \varphi\right\} $ and $K'$ by
\[
K'\cap\R_{+}x=\begin{cases}
\left\{ 0\right\} , & \exists y\in\R_{+}x\,\,\varphi\left(y\right)=\infty\\
\left\{ x\,:\,\varphi\left(x\right)\le1\right\} , & o/w
\end{cases}.
\]
Again, one can check that $\varphi$ is of the form (\ref{eq:Eqform1})
with $L',K'$ in the roles of $L,K$ and that $K'$ is convex. Concluding
the above, we have that if a geometric convex function $\varphi$
satisfies (\ref{eq:EqualityEq}) then its restriction to every ray
$\R_{+}x$ is either linear or a convex indicator, and in particular,
it is both of the form (\ref{eq:Eqform1}) and of form (\ref{eq:EqForm3}).
Conversely, the restriction of each function of the form (\ref{eq:EqForm3})
to every ray is either linear or a convex indicator.\\

We have left to show that such functions indeed satisfy (\ref{eq:EqualityEq}).
To this end, let $\varphi$ satisfy (\ref{eq:EqualityEq}). Then both
$\varphi$ and $\varphi^{\circ}$ are of the form (\ref{eq:Eqform1}).
Suppose 
\[
\varphi=\begin{cases}
\left\Vert x\right\Vert _{K}, & x\in L\\
+\infty, & x\not\in L
\end{cases}
\]
for appropriate convex sets $K,L$. Then, on the one hand we have
that 
\[
\underline{K}_{1/s}\left(\varphi\right)=\left\{ \left\Vert x\right\Vert _{K}\le1/s\right\} \cap L=\left(\frac{1}{s}K\right)\cap L.
\]
On the other hand, since $\varphi^{\circ}$ is of the form (\ref{eq:Eqform1}),
it follows that 
\[
\varphi^{\circ}=\begin{cases}
0, & x\in L^{\circ}\\
\left\Vert x\right\Vert _{K^{\circ}}\left(1-\frac{1}{\left\Vert x\right\Vert _{L^{\circ}}}\right), & x\not\in L^{\circ}
\end{cases}=\begin{cases}
0, & x\in L^{\circ}\\
\left\Vert x\right\Vert _{K^{\circ}}, & x\not\in L^{\circ}
\end{cases}.
\]
Thus 
\[
\underline{K}_{s}\left(\varphi^{\circ}\right)=\conv\left(\left\{ \left\Vert x\right\Vert _{K^{\circ}}\le s\right\} \cup L^{\circ}\right)=\conv\left(\left(sK\right)\cup L^{\circ}\right)
\]
and so $\underline{K}_{s}\left(\varphi^{\circ}\right)=\underline{K}_{1/s}\left(\varphi\right)^{\circ}$
for all $s>0$ as required. Finally, it is not hard to check, that
if $\varphi$ is even, then $\varphi$ must be either a norm or a
convex indicator of a convex set.
\end{proof}
\noindent As a consequence, one also gets the following comparison
of level sets of a function and the level sets of their Legendre transform.
It turns out that this is a special case of \cite[Lemma 8]{FradMeyer2008}:
\begin{cor}
\label{cor:LegendreLevel-Sets}Let $\varphi\in\cvxo.$ Then for every
$s>0$ and every $t>0$ we have
\[
s\cdot\left(\underline{K}_{s}\left(\varphi\right)\right)^{\circ}\sub\underline{K}_{s}\left(\L\varphi\right)\sub\left(s+t\right)\cdot\left(\underline{K}_{t}\left(\varphi\right)\right)^{\circ}.
\]
Moreover, if $s\cdot\left(\underline{K}_{s}\left(\varphi\right)\right)^{\circ}=\underline{K}_{s}\left(\L\varphi\right)$
for all $s>0$ then either $\varphi=\one_{K}^{\infty}$ or $\varphi=\left\Vert \cdot\right\Vert _{K}$
where $K$ is a closed convex set containing the origin.\end{cor}
\begin{proof}
This follows from the fact, pointed out by V. Milman, that the level
sets of $\L\varphi$ and $\varphi^{\circ}$ are connected intimately.
Indeed, note that for every $c>0$ we have
\begin{align*}
\underline{K}_{c}\left(\L\varphi\right)=\left\{ x\in\R^{n}\,;\,\left(\L\varphi\right)\le c\right\}  & =\left\{ x\in\R^{n}\,;\,\forall\, y\in\R^{n},\,\iprod xy-\varphi\left(y\right)\le c\right\} \\
 & =\left\{ x\in\R^{n}\,;\,\forall\, y\in\R^{n},\,\frac{\iprod xy-c}{\varphi\left(y\right)}\le1\right\} \\
 & =\left\{ x\in\R^{n}\,;\,\forall\, y\in\R^{n},\, c\cdot\frac{\iprod{\frac{x}{c}}y-1}{\varphi\left(y\right)}\le1\right\} \\
 & =\left\{ cx\in\R^{n}\,;\,\forall\, y\in\R^{n},\,\frac{\iprod xy-1}{\varphi\left(y\right)}\le\frac{1}{c}\right\} \\
 & =c\cdot\left\{ x\in\R^{n}\,;\,\left(\A\varphi\right)\le\frac{1}{c}\right\} =c\cdot\underline{K}_{\frac{1}{c}}\left(\varphi^{\circ}\right).
\end{align*}
Together with Proposition \ref{prop:level-sets}, the proof is thus
complete.
\end{proof}

\subsection{Log-concave functions with finite positive integral }

In this section we examine some properties of geometric log-concave
functions $f$ satisfying that $0<\int f<\infty$. As later on we
will be connecting the integral of such a function to an integral
over its level sets, it will be helpful to show that its level sets
are convex bodies, that is (closed) convex sets with a finite positive
volume. This will allow us to apply the classical Santal\'{o} (or
reverse Santal\'{o}) inequality for convex bodies in our setting.
To this end, denote the support of a log-concave function $f$ by
${\rm supp}\left(f\right)=\left\{ x:\R^{n}\,:\, f\left(x\right)>0\right\} $.
The following holds. 
\begin{lem}
Let $f=e^{-\varphi}$ be a geometric log-concave function on $\R^{n}$
. Then \label{lem:Integrable} $0<\int f<\infty$ if and only if for
some open ball $B\sub\R^{n}$ and $0<\eps<1$, $\eps\cdot\one_{B}\left(x\right)\le f\left(x\right)$
and for some $r,c>0$, $f\left(x\right)\le e^{-{\textstyle {\scriptstyle c\left|x\right|}}}$
for all $\left|x\right|\ge r$.\end{lem}
\begin{proof}
Assume that $0<\int f<\infty$. The fact that $0<\int f$ obviously
implies that ${\rm supp}\left(f\right)$ is of positive Lebesgue measure.
Since ${\rm supp}\left(f\right)$ is a convex set, and $f$ is continuous
on its support, there exist $0<\eps<e^{-1}$ and an open ball $B\sub{\rm supp}\left(f\right)$
such that $f\left(B\right)\ge\eps$. Denote the unit sphere in $\R^{n}$
by $\Sph^{n-1}=\left\{ \theta\in\R^{n}\,:\,\left|\theta\right|=1\right\} $.
Then for each $\theta\in\Sph^{n-1}$ there exists $R>0$ such that
$f\left(R\theta\right)\le e^{-1}$. Indeed, otherwise we would have
that $f\left(\left\{ \theta r\,:\, r\ge0\right\} \right)>e^{-1}$
for some $\theta\in\Sph^{n-1}$and so, by the log-concavity of $f$,
$f\ge\eps$ on the convex hull of $B$ and the ray $\left\{ \theta r\,:\, r\ge0\right\} $,
which has infinite Lebesgue measure, and so $\int f=\infty$, a contradiction.
Hence, by the compactness of $\Sph^{n-1}$ there exists $r=\min\left\{ R>0\,:\, f\left(\theta R\right)\le e^{-1}\,\text{{\rm for all }}\,\theta\in\Sph^{n-1}\right\} $,
and so for every $x\in\R^{n}$ with $\left|x\right|\ge r$,
\[
e^{-1}\ge f\left(\frac{r}{\left|x\right|}x\right)=f\left(\frac{r}{\left|x\right|}x+\left(1-\frac{r}{\left|x\right|}\right)0\right)\ge f\left(x\right)^{{\textstyle {\scriptstyle r/\left|x\right|}}}f\left(0\right)^{1-r/\left|x\right|}=f\left(x\right)^{r/\left|x\right|}
\]
from which it follows that $f\le e^{-\left|x\right|/r}$ for all $\left|x\right|\ge r$. 

For the opposite direction we have:
\[
0<\int\eps\cdot\one_{B}<\int e^{-\varphi}<\int_{\left|x\right|<r}e^{-\varphi}+\int_{\left|x\right|\ge r}e^{-\varphi}\le\left|rB_{2}^{n}\right|+\int e^{-c\left|x\right|}<\infty
\]
\end{proof}
\begin{rem}
One may check that the above lemma is equivalent to the fact that
a geometric log-concave function $f$ has a finite positive integral
if and only if its support is of full dimension and $f$ does not
attain the constant value $1$ on a whole ray.
\end{rem}
As a consequence of the above lemma we can conclude that the level
sets of a geometric log-concave function $f=e^{-\varphi}$ satisfying
that $0<\int f<\infty$ are convex bodies, and that the class of such
functions is closed under the polarity transform, that is $0<\int e^{-\varphi^{\circ}}<\infty$:
\begin{prop}
\label{prop:LevelSetBodies}Let $f=e^{-\varphi}$ be a geometric log-concave
function on $\R^{n}$, satisfying that $0<\int f<\infty$. Then for
all $t>0$, $\underline{K}_{t}\left(\varphi\right)$ and $\underline{K}_{t}\left(\varphi^{\circ}\right)$
are compact convex sets with non-vanishing Lebesgue volume. Moreover,
$0<\int e^{-\varphi^{\circ}}<\infty$.\end{prop}
\begin{proof}
By Lemma \ref{lem:Integrable}, $f\left(x\right)\le e^{-c\left|x\right|}$
for all $\left|x\right|\ge r$ and $\eps\cdot\one_{B}\le f$ for some
$0<\eps<1$ and an open ball $B$. Let $t>0$ such that $x\in\overline{K}_{t}\left(f\right)$
for some $\left|x\right|\ge r$. Since $f\left(x\right)\le e^{-c\left|x\right|}$
for all $\left|x\right|\ge r$, it follows that $x\in\overline{K}_{t}\left(e^{-c\left|\cdot\right|}\right)$
and so $\left|\overline{K}_{t}\left(f\right)\right|\le\left|\overline{K}_{t}\left(e^{-c\left|\cdot\right|}\right)\right|<\infty$.
By the same reasoning the last inequality holds for all $s<t$. As
$\overline{K}_{t}\left(f\right)$ is monotonically decreasing with
$t$, we have that $\left|\overline{K}_{t}\left(f\right)\right|<\infty$
for all $0<t<1$. In other words $\left|\underline{K}_{t}\left(\varphi\right)\right|<\infty$
for all $t>0$. Since $\eps\cdot\one_{B}\le f$ it follows that $\overline{K}_{t}\left(\eps\cdot\one_{B}\right)\sub\overline{K}_{t}\left(f\right)$,
and so $0<\left|\overline{K}_{t}\left(\eps\cdot\one_{B}\right)\right|\le\left|\overline{K}_{t}\left(f\right)\right|$,
for all $0<t\le\eps$. In other words $0<\left|\underline{K}_{t}\left(\varphi\right)\right|$
for all $a:=\log\left(1/\eps\right)\le t$. As $\varphi$ is a geometric
log-concave function, we have that $\varphi\left(x\right)\le a\cdot\left\Vert x\right\Vert _{\underline{K}_{a}\left(\varphi\right)}$
for all $x\in\underline{K}_{a}\left(\varphi\right)$ which means that
$\underline{K}_{t}\left(a\cdot\left\Vert x\right\Vert _{\underline{K}_{a}\left(\varphi\right)}\right)\sub\underline{K}_{t}\left(\varphi\right)$
for all $0<t<a$ and hence $0<\left|\underline{K}_{t}\left(\varphi\right)\right|$
for all $0<t$. Proposition \ref{prop:level-sets} then implies that
\[
0<\left|\underline{K}_{1/t}\left(\varphi\right)\right|\le\left|\underline{K}_{t}\left(\varphi^{\circ}\right)\right|\le2^{n}\cdot\left|\underline{K}_{1/t}\left(\varphi\right)\right|<\infty.
\]
Denote $g\left(x\right)=e^{-\varphi^{\circ}\left(x\right)}$. Since
$0<\left|\underline{K}_{1}\left(\varphi^{\circ}\right)\right|<\infty$
and $\underline{K}_{1}\left(\varphi^{\circ}\right)$ is convex, there
exist an open ball $B$ and a radius $r>0$ such that $B\sub\underline{K}_{1}\left(\varphi^{\circ}\right)\sub rB_{2}^{n}$,
and so on the one hand 
\[
\frac{1}{e}\cdot\one_{B}\le\frac{1}{e}\cdot\one_{\underline{K}_{1}\left(\varphi^{\circ}\right)}=\frac{1}{e}\cdot\one_{\overline{K}_{1/e}\left(g\right)}\le g=e^{-\varphi^{\circ}}.
\]
On the other hand, for some $c>0$ we have that $\varphi^{\circ}\left(x\right)\ge\left\Vert x\right\Vert _{\underline{K}_{1}\left(\varphi^{\circ}\right)}\ge c\left|x\right|$
for all $x\not\in\underline{K}_{1}\left(\varphi^{\circ}\right)$ and
in particular for all $\left|x\right|\ge r$. In other words, for
all $\left|x\right|\ge r$ we have that $e^{-\varphi^{\circ}}\le e^{-c\left|x\right|}$.
Thus, Lemma \ref{lem:Integrable} implies that $0<\int e^{-\varphi^{\circ}}<\infty$.
\end{proof}

\section{\label{sec:Santaloineq}Santal\'{o} inequality for the polarity
transform}
\begin{prop}
\label{prop:Santalo}Let $\varphi$ be a geometric even convex function
on $\R^{n}$ with $0<\int e^{-\varphi}<\infty$. Then 
\[
\int_{\R^{n}}e^{-\varphi}\cdot\int_{\R^{n}}e^{-\varphi^{\circ}}\le\left(\left|B_{2}^{n}\right|n!\right)^{2}\left(1+\frac{C}{n}\right)
\]
for some universal constant $C>0$ independent of $n$ and $\varphi$.\end{prop}
\begin{proof}
We have that 
\begin{align*}
\int_{\R^{n}}e^{-\varphi} & =\int_{0}^{1}\left|\overline{K}_{t}\left(e^{-\varphi}\right)\right|dt=\int_{0}^{1}\left|\underline{K}_{\ln\left(1/t\right)}\left(\varphi\right)\right|dt=\int_{0}^{\infty}e^{-s}\cdot\left|\underline{K}_{s}\left(\varphi\right)\right|ds.
\end{align*}
Fix $t>0$. By the convexity of $\varphi$, we have that $\varphi\left(x\right)\ge t\cdot\left\Vert x\right\Vert _{\underline{K}_{t}\left(\varphi\right)}$
for all $x\not\in\underline{K}_{t}\left(\varphi\right)$, from which
it follows that $\underline{K}_{s}\left(\varphi\right)\sub$$\underline{K}_{s}\left(t\left\Vert \cdot\right\Vert _{\underline{K}_{t}\left(\varphi\right)}\right)$
for all $s>t$. Thus,
\begin{align*}
\int_{0}^{\infty}e^{-s}\cdot\left|\underline{K}_{s}\left(\varphi\right)\right|ds & =\int_{0}^{t}e^{-s}\cdot\left|\underline{K}_{s}\left(\varphi\right)\right|ds+\int_{t}^{\infty}e^{-s}\cdot\left|\underline{K}_{s}\left(\varphi\right)\right|ds\\
 & \le\left|\underline{K}_{t}\left(\varphi\right)\right|\cdot\int_{0}^{t}e^{-s}ds+\int_{t}^{\infty}e^{-s}\cdot\left|\underline{K}_{s}\left(t\cdot\left\Vert \cdot\right\Vert _{\underline{K}_{t}\left(\varphi\right)}\right)\right|ds\\
 & =\left|\underline{K}_{t}\left(\varphi\right)\right|(1-e^{-t})+t^{-n}\left|\underline{K}_{t}\left(\varphi\right)\right|\cdot\int_{1}^{\infty}e^{-s}\cdot s^{n}\cdot ds\\
 & =\left|\underline{K}_{t}\left(\varphi\right)\right|\left[(1-e^{-t})+t^{-n}n!\right].
\end{align*}
On the other hand, we may Proposition \ref{prop:level-sets} and a
change of variables to obtain that 
\[
\int_{0}^{\infty}e^{-s}\left|\underline{K}_{s}\left(\varphi^{\circ}\right)\right|ds\le\int_{0}^{\infty}e^{-s}\left(st+1\right)^{n}\left|\underline{K}_{t}\left(\varphi\right)^{\circ}\right|ds=\left|\underline{K}_{t}\left(\varphi\right)^{\circ}\right|\int_{0}^{\infty}e^{-s'+\frac{1}{t}}s'^{n}t^{n}ds=e^{\frac{1}{t}}t^{n}\left|\underline{K}_{t}\left(\varphi\right)^{\circ}\right|n!.
\]
By Proposition \ref{prop:LevelSetBodies}, $\underline{K}_{t}\left(\varphi\right)$
and $\underline{K}_{t}\left(\varphi\right)^{\circ}$ are convex bodies.
Moreover, since $\varphi$ is even, $\underline{K}_{1}\left(\varphi\right)$
and $\underline{K}_{1}\left(\varphi\right)^{\circ}$ are centrally
symmetric, and so we may apply the classical Santal\'{o} inequality
together with the above inequalities to obtain that
\begin{align*}
\int_{\R^{n}}e^{-\varphi}\cdot\int_{\R^{n}}e^{-\varphi^{\circ}} & \le\left(\left|B_{2}^{n}\right|n!\right)^{2}\left(\frac{\left(1-e^{-t}\right)e^{\frac{1}{t}}t^{n}}{n!}+e^{\frac{1}{t}}\right)\le\left(\left|B_{2}^{n}\right|n!\right)^{2}e^{\frac{1}{t}}\left(\frac{t^{n}}{n!}+1\right).
\end{align*}
Note that $\left(n(n+1)\right)^{\frac{1}{n+1}}\le3$ and $\frac{n+1}{\sqrt[n+1]{\left(n+1\right)!}}\le e$
for all $n\ge1$, and so, for $t=\sqrt[n+1]{\left(n-1\right)!}$ we
get that 
\[
e^{\frac{1}{\sqrt[n+1]{\left(n-1\right)!}}}\le1+e\left(\left(n-1\right)!\right)^{-\frac{1}{n+1}}=1+e\left(n\left(n+1\right)\right)^{\frac{1}{n+1}}\left(\left(n+1\right)!\right)^{-\frac{1}{n+1}}\le1+\frac{3e^{2}}{n+1}.
\]
Moreover, we have that $e^{\frac{1}{\sqrt[n+1]{\left(n-1\right)!}}}\frac{\left(\sqrt[n+1]{\left(n-1\right)!}\right)^{n}}{n!}\le\frac{e}{n}$,
and so we may conclude that
\[
\int_{\R^{n}}e^{-\varphi}\cdot\int_{\R^{n}}e^{-\varphi^{\circ}}\le\left(\left|B_{2}^{n}\right|n!\right)^{2}\left(1+\frac{C}{n}\right).
\]
\[
\]
\end{proof}
\begin{rem*}
As Joseph Lehec suggested to the authors after reading a first draft
of this note, one may alternatively use Ball's argument (\cite{BallPhD},
see also \cite{FradeliziMeyer07}) to prove essentially the same inequality.
Namely, the fact that for every $\varphi\in\cvxo$, 
\[
\frac{\varphi\left(x\right)+\varphi^{\circ}\left(y\right)}{2}\ge\sqrt{\varphi\left(x\right)\varphi^{\circ}\left(y\right)}=\sqrt{\sup_{z\in\R^{n}}\varphi\left(x\right)\cdot\frac{\iprod zy-1}{\varphi\left(z\right)}}\ge\sqrt{\left(\iprod xy-1\right)_{+}}
\]
implies, by Ball's argument, that 
\[
\int_{\R^{n}}e^{-\varphi}\int_{\R^{n}}e^{-\varphi^{\circ}}\le\left(\int_{\R^{n}}e^{-\sqrt{\left(\left|x\right|^{2}-1\right)_{+}}}\right)^{2}
\]
which is qualitatively the same as our bound. 
\end{rem*}

\section{\label{sec:ReverseSantaloIneq}Reverse Santal\'{o} inequality for
the polarity transform}
\begin{prop}
\label{prop:revSantalo}Let $\varphi$ be a geometric convex function
on $\R^{n}$ with $0<\int e^{-\varphi}<\infty$. Then 
\[
0.7\cdot c^{n}\cdot\left|B_{2}^{n}\right|^{2}\le\int_{\R^{n}}e^{-\varphi}\cdot\int_{\R^{n}}e^{-\varphi^{\circ}}
\]
where $c>0$ is the best constant for the classical reverse Santal\'{o}
inequality for convex bodies.\end{prop}
\begin{proof}
As in the proof of Proposition \ref{prop:Santalo}, we have that $\int_{\R^{n}}e^{-\varphi}=\int_{0}^{\infty}e^{-s}\cdot\left|\underline{K}_{s}\left(\varphi\right)\right|ds$.
By Proposition \ref{prop:LevelSetBodies}, $\underline{K}_{t}\left(\varphi\right)$
and $\underline{K}_{t}\left(\varphi^{\circ}\right)$ are convex bodies
(for which the classical reverse Santal\'{o} applies) and so by Proposition
\ref{prop:level-sets}, Cauchy-Schwarz inequality, and the classical
reverse Santal\'{o} inequality, we obtain that
\begin{align*}
\int_{\R^{n}}e^{-\varphi}\cdot\int_{\R^{n}}e^{-\varphi^{\circ}} & =\int_{0}^{\infty}e^{-s}\cdot\left|\underline{K}_{s}\left(\varphi\right)\right|ds\cdot\int_{0}^{\infty}e^{-u}\cdot\left|\underline{K}_{u}\left(\varphi^{\circ}\right)\right|du\\
 & \ge\int_{0}^{\infty}e^{-s}\cdot\left|\underline{K}_{s}\left(\varphi\right)\right|ds\cdot\int_{0}^{\infty}e^{-u}\cdot\left|\underline{K}_{1/u}\left(\varphi\right)^{\circ}\right|du\\
 & =\int_{0}^{\infty}e^{-s}\cdot\left|\underline{K}_{s}\left(\varphi\right)\right|ds\cdot\int_{0}^{\infty}\frac{e^{-1/t}}{t^{2}}\cdot\left|\underline{K}_{t}\left(\varphi\right)^{\circ}\right|dt\\
 & \ge\left(\int_{0}^{\infty}\sqrt{e^{-s}\cdot\left|\underline{K}_{s}\left(\varphi\right)\right|}\cdot\sqrt{\frac{e^{-1/s}}{s^{2}}\cdot\left|\underline{K}_{s}\left(\varphi\right)^{\circ}\right|}ds\right)^{2}\\
 & \ge c^{n}\cdot\left|B_{2}^{n}\right|^{2}\cdot\left(\int_{0}^{\infty}\frac{e^{-s/2-1/\left(2s\right)}}{s}ds\right)^{2}\\
 & \ge a\cdot c^{n}\cdot\left|B_{2}^{n}\right|^{2}
\end{align*}
where 
\[
a=\left(\int_{0}^{\infty}\frac{e^{-s/2-1/\left(2s\right)}}{s}ds\right)^{2}\ge0.7
\]
may be computed numerically.
\end{proof}

\subsection*{Acknowledgment}

We would like to thank Joseph Lehec for his insightful remarks. We
would also like to thank Matthieu Fradelizi for referring us to the
paper \cite{FradMeyer2008}. Finally we wish to thank Keshet Einhorn
for her useful comments and corrections.

\bibliographystyle{amsplain}
\bibliography{Lc-Santalo}

\end{document}